\documentclass[12pt,a4paper]{amsart}

\usepackage{amsmath, amsfonts, xifthen, latexsym, amssymb, amsthm, amscd}

\usepackage{mathtools}
\mathtoolsset{showonlyrefs}

\usepackage[T1]{fontenc}
\usepackage[utf8]{inputenc}
\usepackage[american]{babel}

\usepackage[usenames,dvipsnames]{color}

\usepackage{tikz}
\usetikzlibrary{arrows}

\usepackage[shortlabels]{enumitem}

\allowdisplaybreaks

\usepackage[perpage]{footmisc}

\usepackage{graphicx}

\usepackage{url}
\usepackage{hyperref}
\hypersetup{colorlinks=true,citecolor=blue,filecolor=blue,linkcolor=blue,urlcolor=blue}

\def\mathclap#1{\text{\hbox to 0pt{\hss$\mathsurround=0pt#1$\hss}}}

\newtheorem{theorem}{Theorem}[section]

\newtheorem*{theorem*}{Theorem}
\newtheorem{lemma}[theorem]{Lemma}
\newtheorem*{lemma*}{Lemma}
\newtheorem{claim}[theorem]{Claim}
\newtheorem{claim*}{Claim}

\newtheorem{prop}[theorem]{Proposition}

\newtheorem{cory}[theorem]{Corollary}

\theoremstyle{definition}
\newtheorem{definition}[theorem]{Definition}
\newtheorem{example}[theorem]{Example}
\newtheorem*{conventions*}{Conventions}

\newtheorem{remark}[theorem]{Remark}
\newtheorem{remarks}[theorem]{Remarks}

\newcommand{\beq}[1][]{ 
    \ifthenelse{\isempty{#1}}{\begin{equation}}{\begin{equation}\label{#1}} 
}
\newcommand{\eeq}{\end{equation}}

\newcommand{\bdefi}[1][]{ 
    \ifthenelse{\isempty{#1}}{\begin{definition}}{\begin{definition}\label{#1}} 
}
\newcommand{\edefi}{\end{definition}}

\newcommand{\bprop}[1][]{ 
    \ifthenelse{\isempty{#1}}{\begin{prop}}{\begin{prop}\label{#1}} 
}
\newcommand{\eprop}{\end{prop}}

\newcommand{\btheo}[1][]{ 
    \ifthenelse{\isempty{#1}}{\begin{theorem}}{\begin{theorem}\label{#1}} 
}
\newcommand{\etheo}{\end{theorem}}

\newcommand{\blemm}[1][]{ 
    \ifthenelse{\isempty{#1}}{\begin{lemma}}{\begin{lemma}\label{#1}} 
}
\newcommand{\elemm}{\end{lemma}}

\newcommand{\bcory}[1][]{ 
    \ifthenelse{\isempty{#1}}{\begin{cory}}{\begin{cory}\label{#1}} 
}
\newcommand{\ecory}{\end{cory}}

\newcommand{\brema}[1][]{ 
    \ifthenelse{\isempty{#1}}{\begin{remark}}{\begin{remark}\label{#1}} 
}
\newcommand{\erema}{\end{remark}}

\newcommand{\bpf}[1][]{ 
    \ifthenelse{\isempty{#1}}{\begin{proof}}{\begin{proof}[#1]} 
}
\newcommand{\epf}{\end{proof}}

\makeatletter
\newtheorem*{rep@theorem}{\rep@title}
\newcommand{\newreptheorem}[2]{%
\newenvironment{rep#1}[1]{%
 \def\rep@title{#2 \ref{##1}}%
 \begin{rep@theorem}}%
 {\end{rep@theorem}}}
\makeatother
\newreptheorem{theorem}{Theorem}
\newreptheorem{lemma}{Lemma}

\newcommand{\be}{{\beta}}

\newcommand{\Om}{{\Omega}}

\newcommand{\eps}{{\varepsilon}}

\newcommand{\De}{{\Delta}}
\newcommand{\ga}{{\gamma}}
\newcommand{\Ga}{{\Gamma}}

\newcommand{\si}{{\sigma}}
\newcommand{\Si}{{\Sigma}}
\renewcommand{\phi}{{\varphi}}

\newcommand\im{\operatorname{im}}

\renewcommand{\ge}{\geqslant}
\renewcommand{\le}{\leqslant}

\newcommand*\mcapinn[2]{\vcenter{\hbox{$\mathsurround=0pt
  \ifx\displaystyle#1\textstyle\else#1\fi\bigcap$}}}

\newcommand*\mcupinn[2]{\vcenter{\hbox{$\mathsurround=0pt
  \ifx\displaystyle#1\textstyle\else#1\fi\bigcup$}}}

\newcommand{\D}{{\texttt{\large D}}}

\newcommand\R{\mathbb R}

\newcommand\N{\mathbb N}

\newcommand{\actson}{\curvearrowright}

\newcommand{\cal}[1]{{\mathcal #1}}

\makeatletter
\newcommand*\if@single[3]{%
  \setbox0\hbox{${\mathaccent"0362{#1}}^H$}%
  \setbox2\hbox{${\mathaccent"0362{\kern0pt#1}}^H$}%
  \ifdim\ht0=\ht2 #3\else #2\fi
  }
\newcommand*\rel@kern[1]{\kern#1\dimexpr\macc@kerna}
\newcommand*\widebar[1]{\@ifnextchar^{{\wide@bar{#1}{0}}}{\wide@bar{#1}{1}}}
\newcommand*\wide@bar[2]{\if@single{#1}{\wide@bar@{#1}{#2}{1}}{\wide@bar@{#1}{#2}{2}}}
\newcommand*\wide@bar@[3]{%
  \begingroup
  \def\mathaccent##1##2{%
    \if#32 \let\macc@nucleus\first@char \fi
    \setbox\z@\hbox{$\macc@style{\macc@nucleus}_{}$}%
    \setbox\tw@\hbox{$\macc@style{\macc@nucleus}{}_{}$}%
    \dimen@\wd\tw@
    \advance\dimen@-\wd\z@
    \divide\dimen@ 3
    \@tempdima\wd\tw@
    \advance\@tempdima-\scriptspace
    \divide\@tempdima 10
    \advance\dimen@-\@tempdima
    \ifdim\dimen@>\z@ \dimen@0pt\fi
    \rel@kern{0.6}\kern-\dimen@
    \if#31
      \overline{\rel@kern{-0.6}\kern\dimen@\macc@nucleus\rel@kern{0.4}\kern\dimen@}%
      \advance\dimen@0.4\dimexpr\macc@kerna
      \let\final@kern#2%
      \ifdim\dimen@<\z@ \let\final@kern1\fi
      \if\final@kern1 \kern-\dimen@\fi
    \else
      \overline{\rel@kern{-0.6}\kern\dimen@#1}%
    \fi
  }%
  \macc@depth\@ne
  \let\math@bgroup\@empty \let\math@egroup\macc@set@skewchar
  \mathsurround\z@ \frozen@everymath{\mathgroup\macc@group\relax}%
  \macc@set@skewchar\relax
  \let\mathaccentV\macc@nested@a
  \if#31
    \macc@nested@a\relax111{#1}%
  \else
    \def\gobble@till@marker##1\endmarker{}%
    \futurelet\first@char\gobble@till@marker#1\endmarker
    \ifcat\noexpand\first@char A\else
      \def\first@char{}%
    \fi
    \macc@nested@a\relax111{\first@char}%
  \fi
  \endgroup
}
\makeatother
\newcommand{\wb}[1]{{\widebar {#1}}}

\newcommand\sscdot{{\cdot}}

\newcommand\sstimes{{\times}}
\newcommand\ssin{{\,\in\,}}

\newcommand\sscap{{\,\cap\,}}

\newcommand{\susbet}{\subset}
\newcommand{\bclai}[1][]{ 
    \ifthenelse{\isempty{#1}}{\begin{claim}}{\begin{claim}\label{#1}} 
}
\newcommand{\eclai}{\end{claim}}
\newcommand\MT{\operatorname{M}}
\newcommand\I{\operatorname{I}}
\newcommand\B{\operatorname{B}}
\renewcommand\P{\mathcal P}
\newcommand\Rel{\operatorname{Rel}}

\newcommand\Var{\operatorname{Var}}
\newcommand\Cl{\operatorname{Cla}}

\begin{document}

\begin{center}\textbf{Note that this manuscript is fully superseded by preprint arxiv:2203.05888}\\[2cm]\end{center}


\title{Borel version of the Local Lemma}
\author[E. Csóka, Ł. Grabowski, A. Máthé, O. Pikhurko, K. Tyros]{Endre Csóka, Łukasz Grabowski, András Máthé, Oleg Pikhurko, Konstantinos Tyros}

\begin{abstract}
We prove a Borel version of the local lemma, i.e.~we show that, under suitable assumptions, if the set of variables in the local lemma has a structure of a Borel space, then there exists a satisfying assignment which is a Borel function. The main tool which we develop for the proof, which is of independent interest, is a parallel version of the Moser-Tardos algorithm which uses the same random bits to resample clauses that are far enough in the dependency graph. 
\end{abstract}
\maketitle

\setcounter{tocdepth}{2}
\tableofcontents

\newcommand{\rnd}{\normalfont{\texttt{rnd}}}
\newcommand{\RND}{\normalfont{\texttt{RND}}}

\section{Introduction}

\vspace{5pt}\textbf{Conventions.} Throughout the article we frequently use the fact that any natural number is equal to the set of all smaller natural numbers. In particular a natural number is not only an element of $\N:= \{0,1,\ldots\}$ but also a subset of $\N$. 

A vertex in a graph is allowed to have at most one self-loop. The \textit{neighbourhood} of a vertex is the set of all vertices connected to that vertex by an edge. The \textit{degree} of a vertex is the cardinality of its neighbourhood. In particular the degree of a vertex with a self-loop and no other edges is equal to $1$. 

\vspace{5pt}\textbf{Local Lemma.} Let us start by recalling a version of the local lemma. The first version of the local lemma was proved by Erd{\H{o}}s and Lov{\'a}sz \cite{MR0382050}. The version we present follows from the subsequent improvement of Lov{\'a}sz (published by Spencer~\cite{MR0491337}). For more historical context we refer to the classical exposition in~\cite{MR2437651}. 

Let $G$ be a graph and let $b$ be a natural number greater than $1$. The elements of the vertex set $V(G)$ should be thought of as variables which can take values in the set $b =\{0,1,2,\ldots,b-1\}$. Let $\mathbf R$ be a function whose domain is $V(G)$ and such that for $x\in V(G)$ we have that $\mathbf R(x)$ is a set of $b$-valued functions defined on the neighbourhood of $x$. Such a function $\mathbf R$ is an example of a \textit{local rule} on $G$. We say that $f\colon V(G) \to b$  \textit{satisfies $\mathbf R$} if for every $x\in V(G)$ the restriction of $f$ to the neighbourhood of $x$ belongs to $\mathbf R(x)$. 

The local lemma gives a condition under which a satisfying assignment exists. For $x\in V(G)$ let $p(x)$ be the probability of failure at $x$, i.e.~$p(x) := 1 - \frac{|\mathbf R(x)|}{b^{\deg(x)}}$, where $\deg(x)$ is the degree of $x$. Let $\Rel(G)$ be the graph whose vertex set is $V(G)$ and such that there is an edge between $x$ and $y$ if and only if the neighbourhoods of $x$ and $y$ are not disjoint (we allow $x$ and $y$ to be equal). Finally,  let $\De$ be the maximal vertex  degree in $\Rel(G)$. 

\begin{theorem}[Lov{\'a}sz's Local Lemma~\cite{MR0491337}]\label{blll} 
If for all $x\in V(G)$ we have $p(x) < \frac{1}{e\De}$ then there exists $f\colon V(G)\to b$ which satisfies $\mathbf R$.
\end{theorem}

In order to motivate our \textit{Borel version} of Theorem~\ref{blll}, let us recall a classical application of the local lemma to colorings of Euclidean spaces from~\cite{MR0382050}. A \textit{$b$-coloring} of $\R^n$ is a function $f\colon \R^n \to b$. We say that a set $U\subset \R^n$ is \textit{multicolored} with respect to a $b$-coloring if $U$ contains points of all $b$ colors.

\begin{cory}[\cite{MR0382050}]\label{kupa}
Let $T\subset \R^n$ be a finite set of vectors and let $b\in \N$ be such that $b(1-\frac{1}{b})^{|T|}< \frac{1}{e\sscdot|T|^2}$. Then there exists a $b$-coloring of $\R^n$ such that for every $x\in \R^n$ the set $x+T$ is multicolored.
\end{cory}

\bpf
We define a graph $G$ as follows. Let $V(G) := \R^n$ and let $(x,y)\in E(G)$ if for some $t\in T$ we have $x+t =y$. We define a local rule $\mathbf R$ by letting $\mathbf R(x)$ be the set of all surjections from the neighbourhood of $x$ to $b$. It is clear that in order to prove the corollary we need to find a function $f\colon V(G)\to b$ which satisfies $\mathbf R$.

It is easy to bound the maximal degree in $\Rel(\cal G)$ by $|T|^2$. Since the image of a non-surjection is contained in a set with $b-1$ elements, we see that for all $x \ssin V(G)$ the probability $p(x)$ of failure is bounded from above by $\frac{b(b-1)^{|T|}}{|b^{T}|} = b(1-\frac{1}{b})^{|T|}$. Thus, we can apply Theorem~\ref{blll} to finish the proof.
\epf

\vspace{5pt}\textbf{Borel Local Lemma.} Our Borel version of the local lemma allows to deduce that the $b$-coloring in Corollary~\ref{kupa} can be demanded to be a Borel function. 

Let $G$ be a graph as before, but let us additionally assume that $V(G)$ is a standard Borel space and that $\mathbf R$ is a \textit{Borel local rule}. Since it is notationally involved to precisely define what it means for a local rule to be Borel, we defer doing it to Section~\ref{mtal}. For now we only assure the reader that in all naturally occurring applications of the local lemma the local rules are in fact Borel. For example, the local rule in the proof of Corollary~\ref{kupa} is Borel.

We say that a graph $G$ is of \textit{uniformly subexponential growth} if for every $\eps>0$ there exists $r$ such that for all $R>r$ and  all vertices $v\in V(G)$ the set of vertices of $G$ at distance at most $R$ from $v$ has cardinality less than $(1+\eps)^R$.

\begin{theorem}[Borel Local Lemma]\label{into}
Let $G$ be graph such that $V(G)$ is a standard Borel space and let $\mathbf R$ be a Borel local rule on $G$. Furthermore let us assume that the graph $\Rel(G)$ is of uniformly subexponential growth, and let $\De$ be the maximal degree in $\Rel(G)$.

If for all $x\in V(G)$ we have $p(x) < \frac{1}{e\De}$ then there exists a Borel function $f\colon V(G)\to b$ which satisfies $\mathbf R$.
\end{theorem}

This theorem follows from a more general Corollary~\ref{cll}. 

In the proof of Corollary~\ref{kupa} the graph $\Rel(G)$ is easily seen to be of uniformly subexponential growth.  Thus we obtain the following corollary of Theorem~\ref{into}.

\bcory[kuba]
In Corollary~\ref{kupa} we can additionally demand the $b$-coloring of $\R^n$ to be a Borel function.
\ecory

\begin{remarks}
\begin{enumerate}[(i), nosep,topsep=0pt,partopsep=0pt, itemsep=1mm, wide]
\item The assumption of the local lemma for example  in \cite[Theorem 1.5]{MR0491337} is that $p(x)$ should be less that $\frac{1}{e(\De+1)}$, as opposed to our $\frac{1}{e\De}$. The difference is only notational, and comes from the fact that we count self-loops when calculating vertex degrees.
\item One can weaken the assumption on $p(x)$ in the local lemma to $p(x)< \frac{(\De-1)^{\De-1}}{\De^\De}$, and this is best possible \cite{MR837067}. Corollary~\ref{cll} shows that the same is true in the case of our Borel local lemma.
\item If the set $V(G)$ of vertices is countable, it can be regarded as a standard Borel space when we equip it with the discrete Borel structure. As will turn out, in this situation all local rules on $G$ are Borel. If $V(G)$ is finite then $\Rel(G)$ is also finite, and hence of uniformly subexponential growth. Thus Theorem~\ref{into} includes Theorem~\ref{blll} as a special case when $V(G)$ is a finite set. 
\end{enumerate}
\end{remarks}

\textbf{The Moser-Tardos algorithm with limited randomness.} The technique we use to prove Theorem~\ref{into} is a modified Moser-Tardos algorithm, and it is of independent interest. 

The Moser-Tardos algorithm (MTA) is a randomized algorithm for \textit{finding} the satisfying assignment under the assumptions of the local lemma. A version of it has been first described by Moser~\cite{MR2780080}, and a modified version has been described by Moser and Tardos~\cite{MR2606086}. We refer to the introduction of~\cite{MR2606086} for the history of attempts to find a constructive version of the local lemma.

To motivate our modified MTA let us recall the parallel version of the MTA. Let us assume that the set  $V(G)$ of vertices is finite. We start by ``sampling'' each point of $V(G)$ at random, i.e.~we choose uniformly at random a function $f_0\in b^{V(G)}$. Now we choose a subset $W_0\subset V(G)$  which is maximal among the subsets of $V(G)$ satisfying the following two properties: 

\begin{enumerate}[(i),nosep]
\item The function $f_0$ violates the local rule at all points of $W_0$. More precisely, for all $x\in W_0$ we have that $f_0$ restricted to the neighbourhood of $x$ is \textit{not} an element of $\mathbf R(x)$.
\item $W_0$ is an independent set in the graph $\Rel(G)$. 
\end{enumerate}
We define $f_1$ by ``resampling $f_0$ at variables in $W_0$''. More precisely, we start by defining $X_0$ to be the union of neighbourhoods of points in $W_0$, and we let  $Y_0 := V(G)\setminus X_0$. Now, we define $f_1$ to be equal to $f_0$ on $Y_0$. Finally, we choose uniformly at random a function in $b^{X_0}$, and we define $f_1$ to be equal to that function on $X_0$. 

We repeat this procedure with $f_1$ in place of $f_0$, and so on, until we end up with a satisfying assignment. With overwhelming probability a satisfying assignment will be found in a time which is linear in $\log(|V(G)|)$.

Let us informally describe how we modify the MTA. We partition $V(G)$ into $p$ disjoint parts $V_0,\ldots, V_{p-1}$ with the property that for every $i\in p$ and all distinct $x,y\in V_i$ we have that $x$ and $y$ are at least $r$ far away from each other in the graph $\Rel(G)$, where the choice of $r$ depends only on the growth of the balls in $\Rel(G)$, but not on $|V(G)|$. Now we assign to each part $V_i$ a ``source of randomness'', i.e.~an element $\rnd$ of $b^\N$, and we use this single sequence for resamplings of \textit{all} points $x$ which are  in $V_i$. 

Thus the points which lie in the same parts are no longer resampled independently from each other. We refer to this version of the MTA as the \textit{MTA with limited randomness} and we describe it precisely in Section~\ref{mtal}. 

\vspace{5pt}
\textbf{Previous results and open questions.} To our best knowledge, there is no previous work which establishes any \textit{Borel} variants of the local lemma. 

However, measurable variants of the local lemma have been studied by Kun~\cite{1303.4982v2} and very recently by Bernshteyn \cite{1604.07349v2} (the present work has been carried out independently of \cite{1604.07349v2}).  Let us discuss those of the results of~\cite{1303.4982v2} and~\cite{1604.07349v2} which are related to the present work.  

We warn the reader that the following definition is not equivalent to the similar notion in \cite{1604.07349v2}.

\begin{definition}
Let $\Ga$ be a countable group, let $X$ be a standard Borel space, let $\nu$ be a probability measure on $X$ and let $\rho\colon \Ga\actson X$ be an action by measure-preserving Borel bijections. For  a sequence $\ga_0,\ldots, \ga_{k-1}$ of elements of $\Ga$ we define $G = G(\ga_0,\ldots, \ga_{k-1})$ to be the graph with $V(G) := X$ and $(x,y) \ssin E(G)$ if for some $i\ssin k$ we have  $\ga_i.x=y$. 

We say that \textit{the measurable local lemma holds for the action $\rho$} if for  all sequences $\ga_0,\ldots, \ga_{k-1}$ and all Borel local rules on $G(\ga_0,\ldots, \ga_{k-1})$ such that $p(x) < \frac{1}{e\De}$ for all $x\in X$, there exists a measurable function $f\colon X \to b$  which satisfies $\mathbf R$. 
\end{definition}

The methods which we use to prove Theorem~\ref{into} can be rather easily modified to prove the following theorem.

\begin{theorem}[Measurable Local Lemma]\label{m-into}
Let $\Ga$ be a countable amenable group, let $X$ be a standard Borel space, let $\nu$ be a probability measure on $X$ and let $\rho\colon \Ga\actson X$ be an action by measure-preserving Borel bijections which is essentially free, i.e.~for $\nu$-almost all $x\ssin X$ we have that the map $\Ga\to X$ given by $\ga\mapsto \ga.x$ is a bijection. Then the measurable local lemma holds for $\rho$. 
\end{theorem}

We do not provide a proof of this theorem in the present paper, but it will be included in a future work. However, if $\Ga$ is of subexponential growth (hence in particular amenable), then Theorem~\ref{m-into} clearly follows from Theorem~\ref{into}.  

In~\cite{1303.4982v2} it is shown that the standard Moser-Tardos algorithm can be applied in the setting of an infinite countable graph. As a corollary, the measurable local lemma holds for the Bernoulli action $\Ga\actson [0,1]^\Ga$, where $\Ga$ is an arbitrary (not necessarily amenable) group . In~\cite{1604.07349v2} it is shown that the meaurable local lemma holds for any action $\Ga \actson X$ for which there exists a $\Ga$-equivariant Borel surjection onto $[0,1]^\Ga$.

In particular, if $\Ga$ is an amenable group then the results of~\cite{1303.4982v2}  and~\cite{1604.07349v2} imply the measurable local lemma only for the actions $\Ga\actson X$ which have infinite entropy. This is a rather large constrain - many natural actions of amenable groups which are covered by Theorems~\ref{into} and Theorems~\ref{m-into} do not have infinte entropy. For example, as far as we know, it is impossible to deduce the measurable, let alone Borel, version of Corollary~\ref{kupa} from  ~\cite{1303.4982v2} or~\cite{1604.07349v2}

On the other hand, the results in~\cite{1303.4982v2}  and~\cite{1604.07349v2} do not require the group $\Ga$ to be amenable. It is very interesting open problem to determine whether the measurable local lemma holds for all probability measure preserving actions of all groups.

Let us make also the following definition.

\begin{definition}
Let $\Ga$ be a countable group, let $X$ be a standard Borel space, and let $\rho\colon \Ga\actson X$ be an action by Borel bijections. As before , for a sequence $\ga_0,\ldots, \ga_{k-1}$ of elements of $\Ga$ we define $G = G(\ga_0,\ldots, \ga_{k-1})$ to be the graph with $V(G) := X$ and $(x,y) \ssin E(G)$ if for some $i\ssin k$ we have  $\ga_i.x=y$. 

We say that \textit{the Borel local lemma holds for the action $\rho$} if for  all sequences $\ga_0,\ldots, \ga_{k-1}$ and all Borel local rules on $G(\ga_0,\ldots, \ga_{k-1})$ such that  $p(x) < \frac{1}{e\De}$ for all $x\in X$, there exists a Borel function $f\colon X \to b$  which satisfies $\mathbf R$. 
\end{definition}

Theorem~\ref{into} establishes in particular that the Borel local lemma holds for all Borel actions of all groups of subexponential growth. In view of Theorem~\ref{m-into} it would be interesting to know whether the Borel local lemma holds for all Borel actions of all amenable groups. In fact, as far as we know, it could be that the Borel local lemma holds for all Borel actions of all groups.

\subsection{Basic notation and definitions}\mbox{}

Given a set $X$, we denote by $\P(X)$  the power set of $X$. Given another set $Y$ we let $X^Y$ be the set of all functions from $Y$ to $X$. We do not automatically identify the power set of $X$ with $2^X$, which is the set of all functions from $X$ to $2=\{0,1\}$. The cardinality of $X$ is denoted by $|X|$. 

As already mentioned, we frequently use the fact that any natural number is equal to the set of all smaller natural numbers. This occasionally leads to the following notational clash: when $p\ssin \N$, the symbol $2^p$ denotes both the set of functions from $p$ to $2$, and the cardinality of that set. However, it should not lead to a confusion.

An oriented graph is a pair $(V,E)$ where $V$ is a set and $E$ is a subset of $V\times V$. In particular, oriented graphs are allowed to have self-loops, but each vertex can have at most one self-loop. An oriented graph $G$ is \textit{symmetric} if for any two vertices $x,y\in V(G)$ we have that $(x,y) \ssin E(G)$ if and only if $(y,x)\ssin E(G)$. The \textit{symmetrization} of an oriented graph $G$ is the smallest symmetric graph $H$ such that $V(G) = V(H)$ and $E(G)\susbet E(H)$.

For $x\in V(G)$ we let $\Var(x) := \{y\ssin V(G)\colon (x,y)\ssin E(G)\}$,  $\Cl(x) :=\{y\ssin V(G)\colon$ $(y,x)\ssin E(G)\}$, and $N(x) := \Var(x) \cup \Cl(x)$. If we need to stress the dependence on $G$ we use the notation $\Var_G(x)$, $\Cl_G(x)$, and $N_G(x)$. If $A$ is a set of vertices then we let $\Var(A)  := \bigcup_{x\in A} \Var(x)$, and similarly for $\Cl(A)$ and $N(A)$.

We define $\Rel(G)$ to be the oriented graph on the same set of vertices as $G$ and with an oriented edge from $x$ to $y$ if and only if $\Var_G(x)\cap \Var_G(y)$ is non-empty. Clearly $\Rel(G)$ is symmetric.

A \textit{variable graph} is an oriented graph $G$ together with two families $a(x)$, $x\ssin V(X)$, and $b(x)$, $x\ssin V(X)$, where $a(x)$ is a well-order on $\Var(x)$ and $b(x)$ is a well-order on $\Cl(x)$.

A \textit{relation graph} is an oriented graph with edges labelled by natural numbers, in such a way that at each vertex $x$ all out-edges starting at $x$ have different labels. 

If $G$ is a variable graph then we consider $\Rel(G)$ to have the following relation graph structure. For every vertex $x$ we label the edges starting at $x$ with distinct elements of $|N_{\Rel(G)}(x)|$, which we do by ordering the elements of $N_{\Rel(G)}(x)$ in the following way. Let $y,z\in N_{\Rel(G)}(x)$ be two distinct elements, let $v$ and $w$  be the smallest elements of, respectively, $\Var_G(x)\cap \Var_G(y)$ and $\Var_G(x)\cap \Var_G(z)$, with respect to the order on $\Var_G(x)$.  If $v$ is smaller than $w$ in $\Var_G(x)$ then we define $y<z$. If $w$ is smaller than $v$ then we define $z<y$. If $v=w$, then we define $y<z$ if and only if $y$ is smaller than $z$ in $\Cl_G(v)$.


\begin{example}
Let us explain how to obtain a variable graph from an instance of $3$-SAT in conjunctive normal form (CNF). An instance of $3$-SAT in CNF consists of a sequence $x_0,x_1,\ldots, x_{n-1}$ of pair-wise distinct variables and a sequence $c_0,c_1,\ldots,c_{m-1}$ of pair-wise distinct \textit{clauses}, i.e.~each $c_i$ is of the form $a_0 \vee a_1 \vee a_2$, where each $a_j$ is equal either to some $x_k$ or $\neg x_k$. We say that such an instance is \textit{satisfiable} if there exists a function $f\colon \{x_0,\ldots, x_{n-1}\}\to \{0,1\}$ such that all the clauses $c_i$ are satisfied.

We associate to such an instance a variable graph $G$ as follows. We let $V(G) := \{c_0,\ldots, c_{m-1}\} \sqcup \{x_0,\ldots, x_{n-1}\}$, and we let $(c_i,x_j)$ be an edge if either $x_j$ or $\neg x_j$ appears in $c_i$.  In particular the graph $G$ is bipartite and every edge is oriented from some $c_i$ towards some $x_j$. The well-orders necessary to finish the definition of a variable graph are induced from the orders $x_0<x_1<\ldots <x_{n-1}$ and $c_0<c_1<\ldots <c_{m-1}$. It is easy to check that $\Rel(G)$ is a graph on the same set of vertices, such that all edges are of the form $(c_i, c_j)$ for some clauses $c_i$ and $c_j$, and $(c_i, c_j)$ is an edge if $c_i$ and $c_j$ share a variable. In particular every vertex $c_i$ has a self-loop.
\end{example}

\section{Moser-Tardos algorithm}\label{mtal}

\subsection{Borel graphs}\mbox{}

Let $\Om$ be a standard Borel space and $\cal B$ be its Borel $\si$-algebra. A \textit{Borel arrow on $\Om$} is a pair $(U,\ga)$, where  $U\ssin \cal B$ and $\ga\colon U\to \Om $ is a Borel isomorphism onto its image. Given a countable sequence $S = ((U_1,\ga_1),(U_2,\ga_2),\ldots)$ of Borel arrows, we associate to it an oriented graph $\cal G(S)$ as follows. We let $V(\cal G(S)) := \Om$ and 
$$
E(\cal G(S)) := \{(x,y)\colon \exists i\ssin \N \text{ such that } x\ssin U_i \text{ and } \ga_i(x) = y\}.
$$  
 
A \textit{Borel graph} is a graph $\cal G$ which is equal to $\cal G(S)$ for some sequence $S$ of Borel arrows on a standard Borel space. 

The following properties are easy to check.
\blemm[obv] 
Let $\cal G$ be a Borel graph. We have the following properties.
\begin{enumerate}[(i),nosep]
\item The set $E(\cal G)$ of edges of $\cal G$ is a Borel subset of $V(\cal G)\times V(\cal G)$. 
\item If $X\subset \Om$ is a Borel set then the graph induced on $X$ from $\cal G(S)$ is also a Borel graph.
\item The symmmetrization of $\cal G$ is a Borel graph.
\item The graph $\Rel(\cal G)$ is a Borel graph.
\item \label{l5} If $A\subset V(\cal G)$ is a Borel set then also the sets $\Var(A)$, $Cl(A)$ and $N(A)$ are Borel.\qed
\end{enumerate}
\elemm

An \textit{independence function for a Borel graph $\cal G$} is a function $\I\colon \cal B \to \cal B$ with the property that for every $X\in \cal B$ the set $\I(X)$ is a maximal subset of $X$ which is independent in $\cal G$. In the appendix we show that every locally finite Borel graph admits an independence function.

Let $S = ((U_1,\ga_1),(U_2,\ga_2),\ldots)$ be a sequence of Borel arrows on a standard Borel space. The Borel graph $\cal G(S)$ becomes a variable graph when we induce the orders on $\Var(x)$ and $\Cl(x)$ from the sequence $S$ in the following way. For two distinct vertices $y,z\in \Var(x)$ we let $y<z$ if and only if there exists $i$ such that $y=\ga_i(x)$ and for all $j<i$ we have $z \neq \ga_j(x)$.  Similarly for two distinct vertices $y,z\in \Cl(x)$ we let $y<z$ if and only if there exists $i$ such that $x=\ga_i(y)$ and for all $j<i$ we have $x \neq \ga_j(z)$. A \textit{Borel variable graph} is a Borel graph with the structure of a variable graph just described.

\subsection{Moser-Tardos algorithm with limited randomness}\label{mtalik}\mbox{}

It is convenient to fix a natural number $b\,{>}\,1$ for the rest of this article.

A \textit{local rule} for a Borel variable graph $\cal G$ is a map $\mathbf R$ whose domain is $V(\cal G)$ and which assigns to each $x\in V(G)$ a subset $\mathbf R(x)$ of $b^{\Var(x)}$. 

Let $\mathbb S$ denote the set of all finite $b$-valued sequences. In other words, the set $\mathbb S$ consists of all functions $f\colon k \,{\to}\, b$, where $k\in \N$. Since for every $x\in V(\cal G)$ the set $\Var(x)$ is well-ordered, we have a canonical bijection $\iota_x\colon |\Var(x)| \to \Var(x)$. Thus a local rule can be canonically identified with a  function $\mathbf R'\colon V(\cal G) \to \cal P(\mathbb S)$, as follows: we let $\mathbf R'(x)$  be the set $\{f\circ\iota_x\colon f\ssin \mathbf R(x)\}$.

We say that $\mathbf R$ is \textit{Borel} if the preimage of every element of $\cal P(\mathbb S)$ under $\mathbf R'$ is Borel.\footnote{This definition makes sense for all Borel variable graphs $\cal G$. However, it is a reasonable definition only when $\cal G$ is locally finite, in which case the image of $\mathbf R'$ is countable.}  The only purpose of defining the function $\mathbf R'$ was to define what it means for $\mathbf R$ to be Borel, and in fact $\mathbf R'$ will not be used again.

Given a function $f\colon V(\cal G) \to b$ and $x\in V(\cal G)$ we let $\wb f(x)\colon \Var(x)\to b$ be the restriction of $f$ to $\Var(x)$. We say that $f$ \textit{satisfies} $\mathbf R$ if for every $x$ we have  $\wb f(x)\in \mathbf R(x)$.

Given $\mathbf R$, the \textit{complementary rule} $\mathbf R^c$ is defined by setting $\mathbf R^c(x) := b^{\Var(x)} \setminus \mathbf R(x)$. For $f\colon V(\cal G)\to b$ we set $\B_{\mathbf R}(f):= \{x\ssin V(\cal G)\colon \wb f(x) \ssin \mathbf R^c(x)\}$.  When $\mathbf R$ is clear from the context we write $\B$ instead of $\B_{\mathbf R}$.

The proof of the following lemma is a tedious but routine calculus exercise, and so we state it without a proof.

\blemm[krem]
Let $\cal G$ be a locally finite Borel variable graph. If $f\colon V(\cal G)\to b$ is a Borel function and $\mathbf R$ is a  Borel local rule for $\cal G$ then $\B_{\mathbf R}(f)$ is a Borel subset of $V(\cal G)$.\qed
\elemm

To avoid cluttering we will write $\I\B(f)$ instead of $\I(\B(f))$ throughout the article.

A \textit{Borel partition} of $V(\cal G)$ is a tuple $\pi = (W_0,\ldots,W_{p-1})$ such that for all $i\in p$ we have that $W_i$ is a Borel subset of $V(\cal G)$, the sets $W_i$ are pair-wise disjoint, and $\bigcup_{i\in p} W_i = V(\cal G)$. Given a Borel partition $\pi = (W_0,\ldots, W_{p-1})$ and $x\in V(\cal G)$ we define $\pi(x)$ to be the unique $i\ssin p$ such that $x\in W_i$.

A \textit{Moser-Tardos tuple} is a tuple $(\cal G, \mathbf R, \I, \pi)$, where $\cal G$ is a Borel variable graph, $\mathbf R$ is a Borel local rule for $\cal G$, $\I$ is an independence function for $\Rel(\cal G)$, and $\pi= (W_0, W_1,\ldots, W_{p-1})$ is a Borel partition of $V(\cal G)$.

We are ready to present a version of the Moser-Tardos algorithm.  Let $(\cal G, \mathbf R, \I, \pi)$ be a Moser-Tardos tuple, let $k \ssin \N$ and  $\rnd\ssin b^{p\times k}$.  The \textit{$k$-step Moser-Tardos algorithm for} $(\cal G, \mathbf R, \I, \pi)$ and $\rnd$ takes as its input a Borel function $f\colon V(\cal G) \to b$ and returns a sequence of Borel functions $\MT_0 := f$, $\MT_1,\ldots, \MT_k \colon V(\cal G)\to b$. Furthermore if $k'>k$  and $\rnd'\in b^{p\times k'}$ extends $\rnd$ in the obvious sense, then the $k'$-step Moser-Tardos algorithm extends the $k$-step Moser-Tardos algorithm in the obvious sense. Thus we will also speak of the \textit{$\N$-step Moser-Tardos algorithm} which outputs an infinite sequence of functions $\MT_0,\MT_1,\MT_2,\ldots$.

It is convenient to fix  $(\cal G, \mathbf R, \I, \pi)$ and $f$, and regard the  Moser-Tardos algorithm as only depending on $\rnd$.  As such, we do not incorporate  $(\cal G, \mathbf R, \I$, $\pi)$ and $f$ in the notation, and we include $\rnd$ in the notation only when it is necessary to stress the dependence on $\rnd$.  The function $\rnd$ should be thought of as the source of randomness for the Moser-Tardos algorithm (see also Remark~\ref{rapa} below). The tuple  $(\cal G, \mathbf R, \I, \pi, f)$ will be called a \textit{long Moser-Tardos tuple}.

We let  $\MT_0 := f$ and we proceed to define $\MT_1 = \MT_1(\rnd),\ldots, \MT_k = \MT_{k}(\rnd)$ inductively.

Let $h_0\colon V(\cal G) \to \N$ be identically equal to $0$. Now suppose that for some $j < k$ the functions $\MT_j$ and $h_j$ are defined. We let
$$
h_{j+1}(x) = 	\left\{\begin{array}{l l}
  		h_j(x) +1& \quad \mbox{if $x\in \Var(\I\B(\MT_j))$,}\\
		h_j(x) & \quad \mbox{otherwise, }\\ 
	\end{array} \right.
$$
and
$$
\MT_{j+1}(x) = 	\left\{\begin{array}{l l}
  		\rnd(\pi(x) ,h_j(x))& \quad \mbox{if $x\in \Var(\I\B(\MT_j))$,}\\
		\MT_{j}(x)& \quad \mbox{otherwise.}\\ 
	\end{array} \right. 
$$

Using Lemmas~\ref{krem} and~\ref{obv}\ref{l5}, it is easy to prove inductively that the functions $\MT_0,\ldots, \MT_k$ are Borel.

If for some $x\ssin V(\cal G)$ we have $h_k(x)\,{>}\, N$ then we say that the $k$-step Moser-Tardos algorithm \textit{resamples $x$ more than $N$ times}. Similarly if for some $k\in \N_+$ and $x\in V(\cal G)$ we have $h_k(x) >N$ then we say that the $\N$-step Moser-Tardos algorithm \textit{resamples $x$ more than $N$ times}.

If $N\ssin \N$ is such that for all $x\ssin V(\cal G)$ we have  $h_k(x) \,\le\, N$ then we say that the $k$-step Moser-Tardos algorithm \textit{resamples each point at most $N$ times.} Similarly when $N\in \N$ is such that for all $k\in \N$ and all $x\in V(\cal G)$ we have $h_{k}(x)\le N$ then we say that the $\N$-step Moser-Tardos algorithm \textit{resamples each point at most $N$ times}.

\brema[rapa] Throughout the whole article, whenever we speak of the probability of an event, the implied  probability space is one of the spaces $b^{p\times k}$, $k\in \N$, or $b^{p\times \N}$, together with the product measure of the uniform measures on the copies of $b$. Thus for example the statement ``With probability $\eta$ the  $\N$-step Moser-Tardos algorithm resamples each point at most $N$ times'' should be read as: ``Let $(\cal G, \mathbf R, \I,\pi)$ and $f$ be fixed. The measure of the set of those $\rnd\in b^{p\times \N}$ such that the  $\N$-step Moser-Tardos algorithm for $(\cal G, \mathbf R, \I, \pi, \rnd)$ with input $f$ resamples each point at most $N$ times is equal to $\eta$''.
\erema

\newcommand{\used}{\textsc{Used}}
\newcommand{\unused}{\textsc{Unused}}

Finally we define two functions, $\used$  and $\unused$, which will be useful in the upcoming analysis of the $k$-step Moser-Tardos algorithm. Both $\used$  and $\unused$ assign a $b$-valued  sequence to each vertex $x\in V(\cal G)$. We define $\used(x)$ to be the sequence 
$$
\rnd(\pi(x),0),\ldots, \rnd(\pi(x),h_{k-1}(x)-1),
$$
and $\unused(x)$ to be the sequence 
$$
\rnd(\pi(x), h_{k-1}(x)),\ldots, \rnd(\pi(x),k-1).
$$

When we need to stress the dependence of $h_k(x)$, $\used(x)$ and $\unused(x)$ on $\rnd$, we write $h_k(\rnd,x)$, $\used(\rnd, x)$ and $\unused(\rnd,x)$, respectively.

Let us set $X(\cal G, \mathbf R) := \{x\in V(\cal G)\colon \mathbf R(x) \neq b^{\Var_{\cal G}(x)}\}$. We will need to deal with the set $X(\cal G, \mathbf R)$  only to reduce the general case of the local lemma to the case when for all $x\in V(\cal G)$ the sets $|Var(x)|$ have the same number of elements.

We finish with the following simple observation.

\blemm[koma]
Let $(\cal G, \mathbf R, \I, \pi,f)$ and $(\wb{\cal G}, \wb{\mathbf R}, \wb{\I}, \wb{\pi},\wb{f})$ be two long Moser-Tardos tuples, such that 
\begin{enumerate}[(i),nosep]
\item $X(\wb{\cal G},\wb{\mathbf R})\subset V(\cal G)$, and $\cal G$ is an induced subgraph of $\wb{\cal G}$, 
\item for all $x\in V(\cal G)$ we have  $\pi(x) = \wb{\pi}(x)$ and $f(x) = \wb{f}(x)$, 
\item for all $x\in V(\cal G)$ and all $g\ssin b^{\Var_{\cal G}(x)}$ we have that $g\ssin \mathbf R(x)$ if and only if $g$ is a restriction of an element of  $\wb{\mathbf R}(x)$,
\item for all Borel subsets $U\subset V(\cal G)$ we have $\I(U) = \wb{\I}(U)$, and
\item for $x\in V(\wb{\cal G})\setminus V(\cal G)$ we have $|\Cl_{\wb{\cal G}}(x)\cap V(\cal G)|\le 1$. 
\end{enumerate}
Let $k,N\in \N$. The probability that the $k$-step Moser-Tardos algorithm for $(\cal G, \mathbf R, \I, \pi, f)$ resamples each point at most $N$ times is equal to the probability that the $k$-step Moser-Tardos algorithm for $(\wb{\cal G}, \wb{\mathbf R}, \wb{\I}, \wb{\pi}, \wb{f})$ resamples each point at most $N$ times. 
\elemm

\bpf

Let $p$ and $\wb{p}$ denote respectively the numbers of parts of the partitions $\pi$ and $\wb{\pi}$. Let  $\rnd \in 2^{p\times k}$ and let $\wb{\rnd}\in 2^{\wb{p}\times k}$ be an extension of $\rnd$, i.e.~for $i\in p$ and $j\in k$ we have $\rnd(i,j) = \wb{\rnd}(i,j)$.

Let $\MT_0,\ldots, \MT_k$ and $\wb{\MT}_0,\ldots, \wb{\MT}_k$ be the sequence of functions produced by the respective Moser-Tardos algorithms. It is clear that for all $i\in k$ and $x\in V(\cal G)$ we have $\MT_i(x) =\wb{\MT}_i(x)$. It is also clear that ${\B}_{\wb{\mathbf R}}(\wb{\MT}_i)$ is contained in $V(\cal G)$. Thus, we have $\wb{\I}\B_{\wb{\mathbf R}}(\wb{\MT}_i) = \I\B_{\mathbf R}(\MT_i)$. 

In particular any point $x\in V(\cal G)$ is resampled exactly the same amount of times by the Moser Tardos agorithms for $(\cal G, \mathbf R, \I, \pi,f)$ and for $(\wb{\cal G}, \wb{\mathbf R}, \wb{\I}, \wb{\pi},\wb{f})$.

For $x \in V(\wb{\cal G}) \setminus V(\cal G)$ we consider two cases. If $\Cl_{\wb{\cal G}}(x)\cap V(\cal G)$ is empty then the Moser-Tardos algorithm on $(\wb{\cal G}, \wb{\mathbf R}, \wb{\I}, \wb{\pi},\wb{f})$ does not resample $x$ at all. If $\Cl_{\wb{\cal G}}(x)\cap V(\cal G)$ is not empty then it contains a unique point $y\in V(\cal G)$. If there is $z\in \Var_{\cal G}(y)$, then clearly $x$ is resampled exactly as many times as $z$. Finally if $\Var_{\cal G}(y)$ is empty then clearly for any $F\colon V(\wb{\cal G}) \to b$  we have $y\notin \B(F)$, and hence again the Moser-Tardos algorithm does not resample $x$ at all. This finishes the proof.
\epf

\section{Preliminaries on landscapes of trees}

\subsection{Landscapes}\mbox{}
\newcommand{\Canvas}{\operatorname{Canvas}}
\newcommand{\Trees}{\operatorname{T}}
\newcommand{\For}{\operatorname{For}}


For $x \in V(G)\times \N$  we let $\wb x$ be the first coordinate of $x$ and we let the \textit{level of $x$}, denoted by  $\ell(x)$, be the second coordinate of $x$.  For $i\in \N$ and $V\subset V(G)\times \N$ we let $V_i$ be the subset of those $x\in V$ for which $\ell(x) = i$.

Given a variable graph $G$ we define a relation graph $\Canvas(G)$ as follows. The set of vertices of $\Canvas(G)$ is $V(G)\times \N$. For every edge $(x,y)$ of the graph $\Rel (G)$  and every $i\in \N$ we add an edge $((x,i), (y,i+1))$ to $\Canvas(G)$, with the same label as the label of $(x,y)$.

A \textit{pseudo-landscape $\cal L$}  is a triple consisting of the following elements.
\begin{enumerate}[(i),nosep]
\item a variable graph $G_{\cal L}$, 
\item a local rule $\mathbf R_{\cal L}$ on $G_{\cal L}$, 
\item a subgraph $\For(\cal L)$  of $\Canvas(G_{\cal L})$ such that at each vertex the in-degree is either $0$ or $1$. 
\end{enumerate}

Note that the graph in (iii) is in fact a forest and it will be referred to as \textit{the forest of $\cal L$}. We let $V(\cal L)$ and $E(\cal L)$ denote respectively the sets of vertices and edges of $\For(\cal L)$, and we let $\Trees(\cal L)$ be the set of connected components of $\For(\cal L)$. For $\tau \ssin \Trees(\cal L)$ we let $\rho(\tau) \ssin V(\tau)$ be the root, i.e.~the unique vertex with minimal level. Finally, we let $\ell(\tau)$  be the \textit{level of $\tau$}, which is by definition equal to $\ell(\rho(\tau))$.

We let $\Canvas(G_{\cal L},k)$ to be the subgraph of $\Canvas(G_{\cal L})$ induced on the set $V(G_{\cal L})\sstimes k$ of vertices. The minimal $k$ such that the forest in (iii) is a subgraph of $\Canvas(G_{\cal L}, k)$ is called the \textit{height of $\cal L$}.

A pseudo-landscape is a \textit{landscape} if for all $i\ssin \N$ and all distinct $x,y\in V(\cal L)_i$ we have that the distance between $\wb x$ and $\wb y$ in $\Rel(G_{\cal L})$ is at least $2$. 

Let $D,\De, \be\in \N$. We say that $\cal L$ is of of type $(D,\De,\be)$ if the maximal out-degree in $G_{\cal L}$ is at most $D$, the maximal degree in $\Rel(G_{\cal L})$ is at most $\De$, and for every $x\in V(G_\cal L)$ we have $ |\mathbf R_{\cal L}^c(x)| \le \be$.

We finish by defining a \textit{decoration} of a landscape $\cal L$ as consisting  of the following data.
\newcommand{\prev}{\textsc{Prev}}

\begin{enumerate}[(1),nosep]
\item a function $\textsc{Final}(\cal L) \in  b^{V(G_{\cal L})}$
\item for each $x \ssin V(\cal L)$ an element $\prev(x)$ of $\mathbf R_{\cal L}^c(\wb x)$. In particular  $\prev(x)$ is a function $\Var(\wb x) \to b$.
\item a function $\pi \colon V(G_{\cal L}) \to \N$.
\end{enumerate}

Let $D,\De, \be,$ $N_1,N_2,p\in \N$. We say that a decorated landscape  $\cal L$ is of type $(D,\De,\be,$ $N_1,N_2,p)$ if it is of type $(D,\De,\be)$ and additionally $|V(G_{\cal L})| \le N_1$, $|V(\cal L)| = N_2$, and $\im(\pi) \subset p$.


\subsection{Grounded landscapes}\mbox{}
\newcommand{\G}{\operatorname{G}}
\newcommand{\Push}{\operatorname{Push}}

Let $\cal L$ be a landscape. We say that $\cal L$ is \textit{grounded} if all trees of $\cal L$ are at level $0$. The purpose of this subsection is to define an equivalence relation on landscapes and to show that every landscape is equivalent to a grounded one. We start with several definitions.

We say that $\cal L$ is \textit{tight} if at least one tree is at level $0$ and for every tree $\tau\in \Trees(\cal L)$ such that $\ell(\tau)>0$ there exist $\si\in \Trees(\cal L)$ different than $\tau$, $x\in V(\si)$ and $y\in V(\tau)$ such that $(x,y)$ is an edge of $\Canvas(G_{\cal L})$. Note that in particular any grounded landscape is tight.

For a vertex $x$ in $\Canvas(G_{\cal L})$ such that $\ell(x)>0$ we let $\Push(x) := (\wb x, \ell(x)-1)$. 
If $G$ is a subgraph of $\Canvas(G_{\cal L})$ such that for all $x\in V(G)$ we have $\ell(x)>0$ then we let $\Push(G)$ be the graph whose set of vertices is $\{\Push(x)\colon x\in V(G)\}$ and whose set of edges is $\{(\Push(x),\Push(y))\colon (x,y)\in E(G)\}$.

If for all $x\in V(\cal L)$ we have $\ell(x)>0$ then we say that $\cal L$ is \textit{pushable} and we define $\Push(\cal L)$ to be the triple consisting of   the graph $G_{\cal L}$, the local rule $\mathbf R_\cal L$, and the graph $\Push(\For(\cal L))$.

For $\tau\in \Trees(\cal L)$ such that $\ell(\tau)>0$ we let $\Push(\cal L,\tau)$ be the triple consisting of the graph $G_{\cal L}$, the local rule $\mathbf R_\cal L$ and the subgraph of $\Canvas(G_{\cal L})$ which is the union of the graph $\Push(\tau)$ and of all the graphs in $\Trees(\cal L)\setminus \{\tau\}$.

We say that $\tau \in \Trees(\cal L)$ is \textit{pushable} if $\tau$ is a tree which ``witnesses that $\cal L$ is not tight'', i.e.~such that  $\ell(\tau)>0$ and for all $\si\in \cal L$ different than $\tau$, and all $x\in V(\si)$, $y\in V(\tau)$ we have that $(x,y)$ is not an edge in $\Canvas(G_{\cal L})$.

The landscapes $\Push(\cal L)$ for a pushable $\cal L$  and $\Push(\cal L,\tau )$, for a pushable $\tau$ will be referred to as \textit{simple pushes}. The following is straightforward to check.

\blemm
A simple push of a landscape is a landscape.\qed
\elemm

Let $(x,y,z)$ be a triple of distinct elements of $V(\cal L)$. We say $(x,y,z)$ is \textit{rebranchable} if $(x,z)\in E(\cal L)$ and $(y,z)$ is an edge of $\Canvas(G_{\cal L})$.

A \textit{rebranching} of $\cal L$ with respect to a rebranchable triple $(x,y,z)$ is a triple consisting of $G_{\cal L}$, $\mathbf R_{\cal L}$ and the subgraph $H$ of $\Canvas(G_\cal L)$ defined as follows. We let $V(H) := V(\cal L)$ and $E(H) = E(\cal L) \setminus\{(x,z)\} \cup \{(y,z)\}$. A tree $\tau \in \Trees(\cal L)$ is \textit{rebranchable} if there exists a rebranchable triple $(x,y,z)$ such that $(x,z)\in E(\tau)$.

We say that a pair $(y,z)$ of distinct elements of $V(\cal L)$ is \textit{joinable} if $z$ is a root of a tree in $\Trees(\cal L)$ and $(y,z)$ is an edge in $\Canvas(G_{\cal L})$. A tree $\tau \in \Trees(\cal L)$ is \textit{joinable} if for some $y\in V(\cal L)$ the pair (y, $\rho(\tau))$ is joinable.

A \textit{joining} of $\cal L$ with respect to a joinable pair $(y,z)$ is a triple which arises from $\cal L$ by adding the edge $(y,z)$ to $\For(\cal L)$.

The following lemma is straightforward to check.

\blemm[operations]
Let $\cal L$ be a landscape.
\begin{enumerate}[(i),nosep]
\item If $(x,y,z)$ is rebranchable triple then the rebranching of $\cal L$ with respect to $(x,y,z)$ is a landscape. 
\item If $(y,z)$ is a joinable pair then the joining of $\cal L$ with respect to $(y,z)$ is a landscape.
\item $\cal L$ is not tight if and only if $\cal L$ is pushable or there exists $\tau\in \Trees(\cal L)$ which is pushable.
\item A tree $\tau \in \Trees(\cal L)$ is pushable if and only if $\ell(\tau)>0$, it is not rebranchable and it is not joinable.\qed
\end{enumerate}
\elemm

We say that two landscapes are \textit{equivalent} if one arises from the other by a sequence consisting of simple pushes, rebranchings and joinings. 

\blemm
Every landscape is equivalent to a grounded landscape.
\elemm

\bpf
For a landscape $\cal L$ we let $\G(\cal L)\subset T(\cal L)$ denote the set of trees at level zero. We prove the lemma by showing that for every landscape $\cal L$ such that $|T(\cal L)| - |\G(\cal L)| >0$  there exists a sequence of simple pushes, rebranchings, and joinings which produces a landscape $\cal K$ for which  $|T(\cal K)| - |\G(\cal K)| < |T(\cal L)| - |\G(\cal L)|$.

Since a simple push does not  change $|T(\cal L)|$ and does not decrease $|G(\cal L)|$, we may assume that $\cal L$ is tight. Suppose that $\tau\in T(\cal L)$ is such that $\ell(\tau)>0$. 

Since rebranching does not change $|T(\cal L)|$ nor $|\G(\cal L)|$ we may assume that $\tau$ is not rebranchable. Thus, by Lemma~\ref{operations}, we may assume that $\tau$ is joinable. However, performing a joining decreases $|T(\cal L)|$ by $1$ and leaves $|G(\cal L)|$ unchanged which finishes the proof.
\epf

We finish by noting that if $\cal L$ is a decorated landscape, then simple pushes, rebranchings and joinings of $\cal L$  inherit the decoration from $\cal L$ in the obvious way.

\subsection{Counting grounded landscapes}\mbox{}

There is an obvious notion of the \textit{isomorphism of  landscapes} $\cal K$ and $\cal L$: it is  a bijection $ V(G_{\cal K}) \to V(G_{\cal L})$ which induces an isomorphism of the graphs $G_{\cal K}$ and $G_{\cal L}$ compatible with the orders on $\Var(x)$ and $\Cl(x)$ and the local rules, and which induces an isomorphism of $\For(\cal K)$ and $\For(\cal L)$. For decorated landscapes we additionally demand that the induced isomorphisms should be compatible with the decorations in the obvious sense.

In this subsection we will count the iso-classes (i.e.~isomorphism classes) of grounded decorated landscapes of a given type. We start by counting subtrees of $\Canvas(G_{\cal L})$. A \textit{$\De$-labelled tree} is an oriented tree such that at each vertex the in-degree is either $0$ or $1$, and such that the edges are labelled by the elements of $\De$ in such a way that at each vertex all the out-going edges have different labels.

\blemm[ghg] 
Let $\De\ge 2$. The number of $\De$-labeled trees with $N$ vertices is bounded by 
$$
\left(\frac{\De^\De}{(\De-1)^{\De-1}}\right)^N.
$$
\elemm

\bpf This proof is based on the arguments from \cite{spencer-note}. Let $P_0 := 0$ and for a natural number $i>0$ let $P_i$ be the number of iso-classes of $\De$-labelled trees with $i$ vertices. Let 
$$
P(X) := \sum_{i=0}^\infty P_i\sscdot X^i,
$$
and let $\rho$ be the radius of convergence of $P(X)$. Clearly it is enough to show that $P(\rho)<1$ and
$$
\rho \ge \frac{(\De-1)^{\De-1}}{\De^\De}.
$$
In fact we will show $P(\rho)\le \frac{1}{\De-1}$

Note that in a finite $\De$-labelled tree there is a unique vertex with in-degree $0$, and we call this vertex the \textit{root}. By considering all the possibilities for the outgoing edges at the root, we arrive at the following equation:
\beq[equa]
P(X) = X(1+ P(X))^\De.
\eeq
We note that there is only one formal power series $P(X)$ with $P(0)=0$ which fulfils the above equation. 

For $i,j\in \N$ let us define the numbers $Q(i,j)\in \N$ as follows. For all $i\ssin \N$ we let $Q(i,0):=0$, and for all $j>0$ we let $Q(i,j)$ to be the number of $\De$-labelled trees with $j$ vertices and such that all vertices are at the distance at most $i$ from the root. Let 
$$
    Q_i(X) =\sum_{j=0}^\infty Q(i,j)X^j.
$$
The following claim is clear from the definitions.
\begin{claim} 
The polynomials $Q_i$ have non-negative real coefficients, and for every $n$ there exists $N$ such that for all $M>N$ we have that the first $n$ coefficients of $Q_M(X)$ are equal to the first $n$ coefficients of $P(X)$.\qed
\end{claim}

As a direct corollary we obtain the following claim.

\begin{claim} If for some $x\in \R$ the sequence $Q_0(x),Q_1(x),\ldots$ converges to $y\in \R$, then the power series $P(X)$ converges at $x$ and we have $P(x) = y$.
\end{claim}

We proceed to use the above claim to establish the convergence of $P(X)$ for $X \le \frac{(\De-1)^{\De-1}}{\De^\De}$. 
It is clear that for any $i$ the sequence $Q(i,0), Q(i,1),\ldots$ is non-decreasing. Hence we see that for any $x\ge 0$  we have $Q_0(x) \le Q_1(x) \le Q_2(x)\le \ldots$. Also, it is easy to see that each polynomial $Q_i(X)$ is non-decreasing for $X\ge 0$. It follows that it is enough to show that the sequence $Q_0( \frac{(\De-1)^{\De-1}}{\De^\De}),Q_1( \frac{(\De-1)^{\De-1}}{\De^\De}),\ldots$ is bounded from above by $\frac{1}{\De-1}$. 

This is clear for $Q_0(\frac{(\De-1)^{\De-1}}{\De^\De})$, since $Q_0(X) = X$. Let us prove by induction that for all $i \in \N$ we have $Q_i(\frac{(\De-1)^{\De-1}}{\De^\De}) \le \frac{1}{\De-1}$. 

Just like in the case of $P(X)$, by considering all the possibilities for the outgoing edges at the root, we see that for all $i\in \N$ we have
$$
    Q_{i+1}(X) =X(1+Q_i(X))^\De.
$$
Thus, using the inductive assumption, we have 
$$
Q_{i+1}\left(\frac{(\De-1)^{\De-1}}{\De^\De}\right)  \le \left(\frac{(\De-1)^{\De-1}}{\De^\De}\right) \left(1+\frac{1}{\De-1}\right)^\De = \frac{1}{\De-1},
$$
which finishes the proof.

\epf

\blemm[l-fff]
Let $\D,\De,\be,N_1,N_2,p\in \N$ and $\De\ge 2$. The number of iso-classes of grounded decorated landscapes $\cal L$ of type $(D,\De,$ $\be,$ $N_1,N_2,p)$ is at most 
\beq[p]
C\cdot N_2^{N_1}\cdot 
\left(\frac{\De^\De}{(\De-1)^{\De-1}}
\sscdot \be\right)^{N_2},
\eeq
where $C$ depends only on $(D, \De,N_1,p)$.
\elemm

\bpf
In fact we will show that the number in question is at most
\beq[q]
N_1\sscdot(N_1+1)^{D\sscdot N_1}\cdot (D!)^{N_1}\cdot (\De!)^{N_1} \cdot 2^{b^D\sscdot N_1}\cdot b^{N_1}\cdot p^{N_1} \cdot N_1^{N_1}  
\cdot N_2^{N_1}
\cdot \left(\frac{\De^\De}{(\De-1)^{\De-1}}\sscdot \be\right)^{N_2}.
\eeq

Let us explain all the terms in \eqref{q}.  To define a pseudo-landscape of type $(D,\De,\be,$  $N_1, N_2,p)$ we need to define the graph $G_{\cal L}$ of max out-degree at most $D$ on at most $N_1$ vertices. There are at most $N_1\sscdot (N_1+1)^{D\sscdot N_1}$ iso-classes of such graphs, because we can define a graph with vertex set $N_1$ and for each element of $N_1$ we give a sequence of $D$ elements, each of which is either an element of $N_1$ or a special symbol signifying that we put no edge (this special symbol is the reason we have $N_1+1$ and not simply $N_1$). We need the additional factor $N_1$ to specify which of the vertex sets $1,2,\ldots, N_1$ is the vertex set of $G_{\cal L}$.

The variable graph structure, i.e.~the choice of orderings of all the sets $\Var(x)$ and $\Cl(x)$, $x\in V(G_{\cal L})$, contributes the factor $(D!)^{N_1}\cdot (\De!)^{N_1}$.

We need a local rule $\mathbf R_{\cal L}$ on $G_\cal L$. If we fix the graph $G_{\cal L}$ with vertex set $N_1$ then there are at most $2^{b^D\sscdot N_1}$ different local rules, because for each vertex $x\in N_1$ we need to specify a family of function in $b^{\Var(x)}$, and there are at most $2^{b^D}$ such families. 

We need to specify the forest of $\cal L$, and we will do it by specifying a sequence of $\De$-labelled trees $T(0),T(1),\ldots,T(N_1-1)$, such that $\sum_{i=0}^{N_1} |V(T(i))| = N_2$. We allow some o the trees $T(i)$ to have empty vertex sets. By basic enumerative combinatorics and Lemma~\ref{ghg} there are at most  
$$
N_2^{N_1}
\cdot \left(\frac{\De^\De}{(\De-1)^{\De-1}}\right)^{N_2}
$$
such sequences.

Since we are counting decorated landscapes, the choice of $\prev(x)$ for each $x\in V(\cal L)$ contributes the factor $\be^{N_2}$.

Finally we need to specify the  missing elements of the decoration, i.e.~an element $\textsc{Final}(\cal L)\in b^{V(G_{\cal L})}$ and $\pi \in p^{V(G_{\cal L})}$ which contribute the factors  $b^{N_1}$ and $p^{N_1}$ respectively.
\epf

\subsection{The $b$-valued sequences encoded by a decorated landscape}\mbox{}
\newcommand{\Seq}{\textsc{Seq}}
\newcommand{\Ass}{\textsc{Asgn}}
\newcommand{\Final}{\textsc{Final}}

Let $\cal L$ be a decorated landscape of height $k$. Let us inductively define \textit{ assignment functions} $\Ass_{k},\Ass_{k-1},\ldots, \Ass_0\colon V(G_{\cal L})\to b$ as follows. First we let $\Ass_k = \Final (\cal L)$. Now suppose that $\Ass_i$ is defined for some $i\ssin \{1,\ldots, k\}$ and let us define $\Ass_{i-1}$. 

Let $y\ssin V(G_{\cal L})$. If there exists no $x \ssin V(\cal L)_{i-1}$ such that $\wb y\ssin \Var(\wb x)$ then we let $\Ass_{i-1}(y) := \Ass_{i}(y)$. Otherwise, since $\cal L$ is a landscape, there is a unique such $x$, and we let $\Ass_{i-1}(y) := \prev(x)(\wb y)$.

Now for each $x\in V(G_{\cal L})$ we define a $b$-valued sequence $\Seq(x)$ as follows. First let $\Seq_0(x)$ be the empty sequence and suppose that $\Seq_i(x)$ is defined for some $i<k$. If for some $y\in V(\cal L)_i$ we have $\wb x\ssin \Var(\wb y)$  then we define $\Seq_{i+1} (x)$ to be the  concatenation of  $\Seq_i(x)$ and the one element sequence consisting of $\Ass_{i+1}(x)$.  Otherwise we let $\Seq_{i+1}(x) := \Seq_i(x)$. Finally, we let  $\Seq(x) := \Seq_k(x)$.

To stress the dependence on $\cal L$ we might write $\Ass_i(\cal L,x)$ and $\Seq(\cal L,x)$ instead of $\Ass_i(x)$ and $\Seq(x)$, respectively. The following lemma is easy to check.

\blemm[equivalent]
Let $\cal K$ and $\cal L$ be equivalent decorated landscapes. For all $x\in V(G_{\cal L}) = V(G_{\cal K})$ we have $\Seq(\cal K,x) = \Seq(\cal L,x)$.\qed
\elemm

\subsection{Landscape restrictions}\mbox{}
\newcommand{\Res}{\operatorname{Res}}
 
Let $\cal L$ be a landscape and let $H$ be a subgraph of $G_{\cal L}$. We consider $H$ to be a variable graph by inducing the order on $N_H(x)$ and $N_G^\text{in}(x)$ from $N_{G_{\cal L}}(x)$ and $N_{G_{\cal L}}^\text{in}(x)$ respectively.   

We proceed to define a new landscape $\Res_{\cal L}(H)$ referred to as the \textit{restriction of $\cal L$ to $H$}. For brevity let us denote $\Res_{\cal L}(H)$ by $\cal K$. We set $G_{\cal K} := H$, and the local rule $\mathbf R_{\cal K}$ is defined as follows. If for $x\in V(H)$ we have $\Var_H(x) = \Var_{G_\cal L}(x)$ then we set $\mathbf R_{\cal K}(x) := \mathbf R_{\cal L}(x)$. Otherwise we let $\mathbf R_{\cal K}(x) := \cal P(b^{\Var_{H}(x)})$.

Finally we define $\Trees(\cal K)$ as follows: for every tree $\tau$ in $T(\cal L)$ we add to $\Trees(\cal K)$ the connected components of the forest $\tau \cap H$. This finishes the definition of the pseudo-landscape $\cal K$. By construction it is a landscape. If $\cal L$ is decorated then we induce a decoration on $\cal K$ from $\cal L$ in the only sensible way.

We will say that the restriction $\Res_{\cal L}(H)$ is \textit{faithful at a vertex $x \in V(H)$} if $\Var_{H} (y) = \Var_{G_{\cal L}}(y)$ for every $y$ such that $x\in \Var_{G_\cal L}(y)$. The following lemma is easy to check.

\blemm[restrict]
Let $\cal K$ be a restriction of $\cal L$ which is faithful at some vertex $x$. Then $\Seq_{\cal K}(x) = \Seq_{\cal L}(x)$. \qed
\elemm

If $V\subset V(G_{\cal L})$ then $\Res_{\cal L}(V)$ denotes the restriction to the graph induced on $V$ from $G_{\cal L}$.

\section{Local Lemma on Borel graphs of subexponential growth}
\subsection{Analysis of the Moser-Tardos algorithm}\mbox{}

Let $\cal G$ be a Borel graph and let $d\colon V(\cal G) \times V(\cal G) \to \N\cup\{\infty\}$ be the metric induced by the graph distance in $\cal G$. For $F\subset V(\cal G)$ and $i\in \N$ we define  $F_{-i}$ to be the set of those $x\in F$ such that for any $y\notin F$ we have $d(x,y)\ge i$. Furthermore, for $x\in V(\cal G)$ and $r\in \N$ we let $B(x,r):= \{y\in V(\cal G)\colon d(x,y) \le r\}$ be the ball of radius $r$ around $x$. We say that a partition $\pi$ of $V(\cal G)$ is \textit{$r$-sparse} if for every $x\ssin V(\cal G)$ the different points of $B(x,r)$ belong to different parts of $\pi$.  

The great majority of the work needed to obtain a Borel version of Local Lemma is contained in the following theorem. Recall that  $X(\cal G, \mathbf R)$  is defined to be the set $\{x\in V(\cal G)\colon \mathbf R(x) \neq b^{\Var_{\cal G}(x)}\}$, i.e. it is the subset of $V(\cal G)$ where the rule $\mathbf R$ is non-trivial. 

\btheo[main]
Let $(\cal G, \mathbf R,\I,\pi,f)$ be a long Moser-Tardos tuple such that $\sup_{x\in V(\cal G)} N_{\cal G}(x) < \infty$. Let us denote $D:= \sup_{x\in X(\cal G, \mathbf R)} |\Var(x)|$,  $\De := \sup_{x\in V(\cal G)} |N_{\Rel(\cal G)}(x)|$, and $\be := \sup_{x\in \Om}  |\mathbf R^c(x)|$, and let us assume that there are $\eps\in (0,1)$ and $n\in \N$ be such that the following conditions hold.
\begin{enumerate}[1., nosep]
\item For any bounded Borel function  $g\colon V(\cal G) \to \N$  supported on $X(\cal G,\mathbf R)$ there exists a finite set $F\subset V(\cal G)$ which is contained in $B(y,n)$ for some $y\in V(\cal G)$, such that
\beq[jaba]
\max_{x\in V(\cal G)}g(x) \,\,\le\,\, \sum_{x\in F} g(x) \,\, < \,\,(1+\eps)\,\sscdot \sum_{x\in F_{-3}} g(x).
\eeq

\item  We have $\De\ge 2$ and 
\beq[good]
    b^{(1-\eps)\sscdot D} > \frac{\De^\De}{(\De-1)^{\De-1}}\sscdot \be,
\eeq

\item The partition $\pi$ is $n$-sparse.

\end{enumerate}

Then there exist constants $K>0$ and $M\in \N$ such that for all $k\in \N$ and all $N>M$, the probability that the $k$-step Moser-Tardos algorithm resamples some point more than $N$ times is bounded by $b^{-NK}$.

\etheo

\brema \begin{enumerate}[(i),nosep,topsep=0pt,partopsep=0pt, itemsep=0mm, wide]
\item As explained in Remark~\ref{rapa}, the implied probability space in the theorem above is the space $b^{p\times k}$ with the uniform probability measure.

\item We also note that $\sup_{x\in V(\cal G)} N_{\cal G}(x) < \infty$ implies $\be, D, \De <\infty$.
\end{enumerate}
\erema

\bpf[Proof of Theorem~\ref{main}]
We start by proving the following reduction.
\begin{claim}
Without any loss of generality we may assume that for all $x \in X(\cal G, \mathbf R)$ we have $|\Var(x)| = D$.
\end{claim}
\bpf[Proof of Claim] If this is not the case then we construct a new long Moser-Tardos tuple  $(\wb{\cal G}, \wb{\mathbf R}, \wb{\I}, \wb{\pi},\wb{f})$ as follows.

We define $V(\wb {\cal G}) := V(\cal G) \cup (V(\cal G)\sstimes\De)$, and 
$$
E(\wb{ \cal G}) := E(\cal G) \cup \{(x,(x,i))\colon x\in V(\cal G), i\ge |\Var_{\cal G}(x)|\}.
$$
It is easy to see that $\wb{\cal G}$ is a Borel graph. It is also easy to see that $V(\cal G)$ and $V(\cal G)\sstimes\De$ are disjoint sets.\footnote{To prove it one needs to fix a definition of an ordered pair. We use the most standard definition due to Kuratowski. An alternative would be to explicitly define $V(\wb{\cal G})$ to be a disjoint union of $V(\cal G)$ and $V(\cal G)\sstimes \De$. } It follows that $\cal G$ is an induced  subgraph of $\wb{\cal G}$, and by construction we see that for all $x\in V(\cal G)$ we have $|\Var_{\wb{\cal G}}(x)| = D$.

Let $\wb{\mathbf R}$ be the Borel local rule for $\cal G$ defined as follows. For $x\in V(\cal G)$ we let $f \in \wb{\mathbf R}(x)$ if and only if the restriction of $f$ to $\Var_{\cal G}(x)$ is an element of $\mathbf R(x)$. For $x\notin V(\cal G)$ we let $\wb{\mathbf R}(x)$ be the unique empty function (note that for $x\notin V(\cal G)$ the set $\Var_{\wb{\cal G}}(x)$ is empty).

Note that $\Rel(\cal G)$ is an induced subgraph of $\Rel(\wb{\cal G})$, and that its complement is a Borel graph with no edges. This allows to easily extend the independence function $\I$ for $\Rel(\cal G)$ to an independence function $\wb{\I}$ for $\Rel(\wb{\cal G})$.

Recall that $p$ denotes the number of parts in the partition $\pi$. We extend the partition $\pi$ of $V(\cal G)$ to a partition $\wb\pi$ of $V(\wb{\cal G})$ by adding to it sets $W_j\times \{i\}$, for $j\in p$ and $i\in \De$, in an arbitrary order.

It is straightforward to see that the tuple $(\wb{\cal G}, \wb{\mathbf R}, \wb{\I}, \wb{\pi},\wb{f})$ fulfils all the conditions of Theorem~\ref{main}, and additionally for all $x\in X(\cal G, \mathbf R)$ we have $|\Var_{\wb{\cal G}}(x)| = D$.  Now the statement of the claim follows from Lemma~\ref{koma}.
\epf

Thus from now on we assume that for all $x \in X(\cal G, \mathbf R)$ we have $|\Var(x)| = D$. Let us define, for all $k\in\N$ and all $\rnd\ssin b^{p\times k}$, a decorated landscape $\cal L = \cal L(\rnd)$ using the functions $\MT_0(\rnd),\ldots, \MT_k(\rnd)$ given by the $k$-step Moser-Tardos algorithm.


We let $G_{\cal L} := \cal G$, $\mathbf R_{\cal L} := \mathbf R$, and the forest in $\Canvas(G_{\cal L})$ is defined as follows.  For $i=0,\ldots, k-1$ we let $V(\cal L)_i := \I\B(M_i)\sstimes\{i\}$. We claim that for $i\in k{-}1$ and $x\in V(\cal L)_{i+1}$ there exists  $y\in V(\cal L)_i$ such that $\wb x$ and $\wb y$ are connected in $\Rel(\cal G)$. Indeed, otherwise it is easy to check that $x$ must be in $B(M_i)$ which contradicts the maximality property of $IB(M_i)$. 

Now given $x\ssin V(\cal L)_{i+1}$ let $z$  be the vertex in $V(\cal L)_i$ such that $\wb z \le \wb y$ for all $y\in V(\cal L)_i$ for which $\wb x$ and $\wb y$ are connected in $\Rel(\cal G)$. In particular there is an edge from $z$ to $x$ in $\Canvas(G_{\cal L})$ and we add this edge to our forest.

The decoration is defined as follows. The function $\pi\colon V(\cal G)\to  \N$ is already defined; we let $\Final(\cal L) := \MT_k$ and for $x = (\wb x, \ell(x))\in V(\cal L)$  we let $\prev(x)$ to be the function $\MT_{\ell(x)}$ restricted to $\Var(x)$.

The following claim is easy to check.

\begin{claim}
For $x\in V(\cal G)$ we have $\Seq_{\cal L(\rnd)}(x) = \used(\rnd, x)$.\qed
\end{claim}

Let $g(\rnd)\colon V(\cal G)\to b$ be the function $x\mapsto |V(\cal L)\,\,\cap\,\, (\{x\} \sstimes\N)|$. It is routine to check that $g(\rnd)$ is Borel. Furthermore, clearly for $\rnd\in b^{p\times k}$ we have that $g(\rnd)$ is bounded by $k$. Let $F(\rnd)\subset V(\cal G)$ be the set guaranteed for the function $g(\rnd)$ by the first assumption in the statement of Theorem~\ref{main}.

\newcommand{\concat}{\operatorname{concat}}

Recall that $\mathbb S$ is the set of all finite $b$-valued sequences.  Let $L$ be the set of (the iso-classes of) the grounded landscapes, including the empty landscape. The length of $s\in \mathbb S$ will be denoted by $|s|$. If $h\colon p\to \mathbb S$ then we let $\concat(h)$ be the concatenation of the sequences $h(0),h(1),\ldots, h(p-1)$. 

We proceed to define an injective map  
$$
\zeta\colon b^{p\times k} \to \cal P(p) \times \mathbb S \times L.
$$
If $|V(\cal L(\rnd))|$ is empty then we let 
$$
\zeta(\rnd):=(\emptyset,\concat(\rnd),\text{empty landscape}).
$$ 

Otherwise we let $\cal K(\rnd)$ be a grounded landscape equivalent to $\Res_{\cal L(\rnd)}(F(\rnd))$, and we define
$$
\zeta(\rnd):= (\pi(F(\rnd)_{-2}), \concat(U(\rnd)), \cal K(\rnd)),
$$
 where $U(\rnd)$ is defined as follows: if $i \in \pi(F(\rnd)_{-2})$ then we let $U(\rnd)(i)$ be the sequence $\unused(\rnd,x)$, and otherwise we let $U(\rnd)(i)$ be the sequence $\rnd(i)$.

Let us argue that $\zeta$ is indeed an injection. First note that for any set $F$ the restriction of  $\cal L(\rnd)$ to $F$ is faithful at all vertices in $F_{-2}$. By Lemmas~\ref{equivalent} and~\ref{restrict} It follows that if $y\in F(\rnd)_{-2}$ then 
$\Seq_{\cal K(\rnd)}(y)$ is  equal to $\used(\rnd,y)$.

Let us argue that  we can recover the function $U(\rnd)$ from the triple $(\pi(F(\rnd)_{-2})$, $ \concat(U(\rnd))$, $\cal K(\rnd))$. It is enough to show that we can recover the lengths $|U(\rnd)(i)|$ for $i\in p$ from the pair  $(\pi(F(\rnd)_{-2}), \cal K(\rnd))$. This is possible since since $\pi$ assigns different values to all the vertices of $G_{\cal K(\rnd)}$. In particular knowing $\pi(F(\rnd)_{-2})$ allows us to deduce which vertices of $G_{\cal K(\rnd)}$ belong to $F(\rnd)_{-2}$, and so we can recover the lengths $|U(\rnd)(i)|$ as follows:
$$
|U(\rnd)(i)| = \left\{ 
	\begin{array}{l l}
  		k - |\Seq_{\cal K(\rnd)}(y)| & \,\, \mbox{if $i \in \pi(F(\rnd)_{-2})$,}\\
		k & \,\, \mbox{otherwise}\\ 
	\end{array} \right. 
$$

Now it is easy to recover $\rnd$. For $i\,{\notin}\,  \pi(F(\rnd)_{-2})$ we have $\rnd(i) = U(\rnd)(i)$, and for $x\ssin F(\rnd)_{-2}$ we have that $\rnd(\pi(y))$ is a concatenation of $\unused(\rnd, y) = U(\rnd)(\pi(y))$ and $\used(\rnd,y) = \Seq_{\cal K(\rnd)}(y)$. This establishes that  the map $\zeta$ is an injection.

Let us denote $N_1 := \max_{y\in V(\cal G)} |B(y,n)|$. Since we assume $\sup_{x\in V(\cal G)} N_{\cal G}(x) < \infty$, we have that $N_1$ is finite.

\begin{claim}\label{l-sds}
Let $N_2\in \N$. Let $\cal K$ be a grounded landscape of type $(D,\De,\be,  N_1,N_2,p)$ and let $A\in \cal P(p)$. The number of points in $\im(\zeta)$ with the first and  the third coordinate equal to, respectively, $A$ and $\cal K$ is bounded by
$$
 b^{pk-(1-\eps)N_2\sscdot D}.
$$
\end{claim}
\bpf[Proof of Claim]
We are counting the possible sequences $\concat(U(\rnd))$ for those $\rnd\in b^{p\times k}$ such that $\pi(F(\rnd)_{-2}) = A$ and $\cal K(\rnd) = \cal K$. 
 
Clearly $\concat(U(\rnd))$ is a $b$-valued sequence of length equal to $p\sscdot k$ minus ``the number of bits used by the $k$-step Moser-Tardos algorithm for resampling points in $F(\rnd)_{-2}$'', i.e.
$$
\sum_{x\in  F(\rnd)_{-2}} h_k(x).
$$
The latter sum is easily seen to be at least $|V(\cal L)\,\,\cap\,\, (F(\rnd)_{-3})\sstimes\N)|$ times $D$. This is because whenever we have a point $y\in V(\cal L)\,\,\cap\,\, (F(\rnd)_{-3}\sstimes \N)$, we have that in the Moser-Tardos algorithm's passage from $\MT_{\ell(y)}$ to $\MT_{\ell(y)+1}$ all the elements of $\Var(\wb y)$ got resampled, and clearly $\Var(\wb y)\subset F(\rnd)_{-2}$.

The number of points in $V(\cal L)\,\,\cap\,\, (F(\rnd)_{-3}\sstimes\N)$ is, by the inequality~\eqref{jaba},  at least $\frac{1}{1+\eps}N_2 \ge (1-\eps)N_2$.  Thus all in all we see that $\concat(U(\rnd))$ is a $b$-valued sequence of length at most $pk - (1-\eps)N_2\sscdot D$, which establishes the lemma.
\epf

For $N_2\in \N$ let $P(N_2)$ be the probability that $\rnd\in b^{p\times k}$ is such that $V(\cal K(\rnd)) = N_2$.  Since $\zeta$ is an injection we have that $P(N_2)\sscdot b^{p\times k}$ is equal to the number of elements in $\im(\zeta)$ whose fourth coordinate is a grounded landscape of type $(D,\De,\be, N_1,N_2,p)$. This, together with Lemma~\ref{l-fff} and Claim~\ref{l-sds} allows us to bound $P(N_2)$  from above. After taking base-$b$ logarithms, denoted by $\log$,  we obtain 
\begin{multline}
\log(P(N_2)\cdot b^{pk}) = \log(P(N_2)) + pk  \le\\
\le   pk-(1-\eps)N_2\sscdot D 
+  \log(C) + N_1\log(N_2) + N_2 \log (\frac{\De^\De}{(\De-1)^{\De-1}}\sscdot \be)
\end{multline}
where $C$ is the constant from Lemma~\ref{l-fff}. In particular, $C$ does not depend on $N_2$. 

After subtracting $pk$ from both sides and  dividing both sides by $N_2$, we obtain
$$
\frac1{N_2}\log(P(N_2)) \le \frac{\log(C) + N_1\log(N_2)}{N_2} - (1-\eps)D + \log(\frac{\De^\De}{(\De-1)^{\De-1}}\sscdot \be),
$$
Note that the term $\frac{\log(C) + N_1\log(N_2)}{N_2}$ converges to $0$ as $N_2\to \infty$. Because of that and the assumption~\eqref{good}, we see that there exist constants $\wb K>0$ and $M \ssin \N$, both independent of $N_2$, such that for all $N_2>M$ we have
$$
\log(P(N_2)) \le -N_2\sscdot \wb K,
$$
and so $P(N_2) \le b^{-N_2\sscdot\wb K}$.

Let $Q(N_2)$ be the probability that $\rnd\in b^{p\times k}$ is such that $|V(\cal K(\rnd))| > N_2$. 
Clearly $Q(N_2) = \sum_{n_2 > N_2} P(n_2)$, and after making $\wb K$ smaller if necessary we see that for $N_2>M$ we have $Q(N_2) \le b^{-N_2\sscdot \wb K}$. 

The claim of the theorem now easily follows: Let $G\subset b^{p \times k}$ be the set of those $\rnd$ such that $V(\cal K(\rnd)) >N_2$. We just showed that the measure of $G$ is at most $b^{-N_2\sscdot \wb K}$. We claim that if $\rnd\notin G$ then the $k$-step Moser-Tardos algorithm  resamples each point less than $\De \sscdot N_2$ times.

Indeed, let us assume that this is not the case and let $x\in V(\cal G)$  be a point which is resampled at least $\De \sscdot N_2$ times. However, since $\De = \sup_{x\in V(\cal G)} |N_{\Rel(\cal G)}(x)|$, we see that in particular there are at most $\De$ points $y\in V(\cal G)$ such that $x\in \Var(y)$. It follows that for one of these points $y$ we have $|V(\cal L(\rnd))\,\,\cap\,\, \{y\}\sstimes\N| \ge N_2$. Hence, by by the choice of the set $F(\rnd)$,  we have $|V(\cal K(\rnd))|\ge N_2$, which is a contradiction with the fact that $\rnd\notin G$.

Thus in order to finish the proof of Theorem~\ref{main} we would let $K := \frac{\wb K}{\De}$ if we only cared about numbers $N$ which are multiples of $\De$. It is straightforward to see that taking $K:= \frac{\wb K}{2\De}$ works for all $N$. 
\epf

\subsection{Borel Local Lemma}\mbox{}

Let us start by showing that the first item of Theorem~\ref{main} is automatically fulfilled for arbitrary $\eps$ and large enough $n$ if $\cal G$ is of uniformly subexponential growth.

\blemm[growth]
Let $\eps>0$ and let us assume that $n\in\N$ is such that for all $x\in V(\cal G)$ we have $B(x,3n)<(1+\eps)^n$. Let $g\colon V(\cal G) \to \N$  be a bounded Borel function which assumes its maximum at some point $y\in V(\cal G)$. Then there exists $r\in\{3,\ldots, 3n\}$ such that 
\beq[jaba]
\sum_{x\in B(y,r)} g(x)  < (1+\eps)\cdot \sum_{x\in B(y,r-3)} g(x).
\eeq
\elemm

\bpf
Suppose that for all $r\in \{3,\ldots,3n\}$ the inequality \eqref{jaba} does not hold. Then 
$$
\sum_{x\in B(y,3n)} g(x) \ge (1+\eps)^n\sscdot g(y).
$$

Since $B(y,3n)<(1+\eps)^n$, we deduce that for some $z\in B(x,3n)$ we have $g(z)>g(y)$, which contradicts the maximality of $g(y)$.
\epf

We are ready to give a criterion when the functions returned by the Moser-Tardos algorithm converge pointwise to a  Borel function.

\btheo[bll]
Let $(\cal G, \mathbf R,\I,\pi,f)$ be a long Moser-Tardos tuple such that $\cal G$ has a uniformly subexponential growth. Let $\De$ be the maximal degree in $\Rel(\cal G)$,  and let us assume that $\De\ge 2$ and that for all $x\ssin V(\cal G)$ we have
\beq[bank]
1-\frac{|\mathbf R(x)|}{|b^{\Var(x)}|} < \frac{(\De-1)^{\De-1}}{\De^\De}.
\eeq

If $\pi$ is $m$-sparse for sufficiently large $m$ then there exists a constant $K>0$ such that the probability that the $\N$-step Moser-Tardos algorithm for $(\cal G, \mathbf R, \I, \pi, f)$ resamples some point more than $N$ times is bounded by $b^{-NK}$. 

In particular, the functions $\MT_0, \MT_1,\ldots$ returned by the Moser-Tardos algorithm converge with probability $1$ to a Borel function which satisfies $\mathbf R$.
\etheo

\bpf
Let us first explain the ``In particular'' part. Note that the functions $\MT_i$, $i\in \N$ are bounded Borel functions. Thus if their pointwise limit exists, it is  automatically  a Borel function.

By the main part of the theorem we have that with probability $1$ there exists a constant $N$ such that the $\N$-step Moser-Tardos Algortihm resamples at most $N$ times. Since $\cal G$ is of uniformly subexponential growth, in particular balls of finite radii in $V(\cal G)$ are finite.

Thus for every $x\in V(\cal G)$ there exists $i$ such that for all $j\ge i$ we have that $B(x,4)\cap \I\B(\MT_j) =\emptyset$. This shows that in fact $B(x,2)\cap \B(\MT_j) = \emptyset$ for all $j\ge i$. This implies that for $j\ge i$ we have that the restriction of $\MT_j$ and $\MT_i$ to $B(x,2)$ are equal. This together with the fact that $x\notin \B(\MT_j)$ for $j\ge i$ establishes the ``In particular'' part.

\begin{claim*} In order to prove the theorem it is enough to establish the existence of a constant $K>0$ with the property that for every $k$ the probability that there exists $x\in V(\cal G)$ such that the $k$-step Moser-Tardos algorithm resamples $x$  more than $N$ times is bounded by $b^{-NK}$. 
\end{claim*}
\bpf[Proof of Claim]
For $i\in \N$ the product measure on on $b^{p\times i}$ will be denoted by $\mu_i$, and the product measure on $b^{p\times \N}$ will be denoted by $\mu$. 
Let $\pi_i\colon b^{p\times \N}\to b^{p\times i}$ be the restriction to the first $i$ coordinates. 

Let $U_i\susbet b^{p\times i}$ be the set of those $\rnd$ for which the $i$-step Moser-Tardos algorithm resamples each point at most $N$ times. Let us assume that for all $i\in \N$ we have $\mu_b(U_{i}) \ge 1-b^{-NK}$.

Let $V_i := \pi_a^{-1}(U_i)$. Note that we have $\mu(V_i) \ge 1 - b^{NK}$ and for $j>i$ we have $V_j\subset V_i$. It follows that $\mu(\bigcap_{i\in \N} V_i)\ge 1-b^{NK}$ and the claim follows since for any $\rnd\in \bigcap_{i\in \N} V_i$ the $\N$-step  Moser-Tardos algorithm resamples each point at most $N$ times.
\epf

Thus we proceed to show the existence of a constant $K>0$ such that for every $k$ the probability that the $k$-step Moser-Tardos algorithm resamples some point more than $N$ times is bounded by $b^{-NK}$. 

Let $D:= \sup_{x\in V(\cal G)} |\Var(x)|$  and $\be := \sup_{x\in \Om}  |\mathbf R^c(x)|$. Since we assume \eqref{bank}, we also have  
$$
\frac{\be}{b^{D}} < \frac{(\De-1)^{\De-1}}{\De^\De}.
$$

Thus we can fix a number $\eps\in(0,1)$ such that 
\beq[good]
    b^{(1-\eps)\sscdot D} > \frac{\De^\De}{(\De-1)^{\De-1}}\sscdot \be.
\eeq

Since $\cal G$ is of uniformly subexponential growth, we can fix a number $n\in \N$  such that 
\beq[dupa]
|B(x,3n)| < (1+\eps)^n.
\eeq

Let us assume that $\pi$ is $3n$-sparse. We claim that all the conditions of Theorem~\ref{main} are fulfilled with $\eps$ and $3n$. Indeed, the first item of Theorem~\ref{main} follows from Lemma~\ref{growth} and the second one follows from the inequality~\eqref{good}. Thus we can apply Theorem~\ref{main}, which finishes the proof of Theorem~\ref{bll}.
\epf

\begin{cory}[Borel Local Lemma]\label{cll}
Let $\cal G$ be a Borel variable graph and let $\mathbf R$ be a Borel local rule on $\cal G$. Furthermore let us assume that the graph $\Rel(\cal G)$ is of uniformly subexponential growth, and let $\De$ be the maximal degree in $\Rel(\cal G)$. If for all $x\in V(\cal G)$ we have
$$
1-\frac{|\mathbf R(x)|}{|b^{\Var(x)}|} < \frac{(\De-1)^{\De-1}}{\De^\De}.
$$
then there exists a Borel function $f\colon V(G)\to b$ which satisfies $\mathbf R$.
\end{cory}

\begin{proof}

By Lemma~\ref{dupa2} in the appendix, there exists an independence function for $\Rel(\cal G)$ so we let $\I$ be such an independence function. Let $m\in \N$ be as in the previous theorem. By Lemma~\ref{dupa3} om the appendix, we can let $\pi$ to be an $m$-sparse partition of $V(\cal G)$. Finally let $f\colon V(\cal G) \to b$ be a constant function equal to $0$ everywhere.

If $\De=1$ then it means that the sets $\Var(x)$, $x\in V(\cal G)$, are pairwise disjoint. In this case it is very easy to see that the functions $\MT_0, \MT_1,\ldots$ returned by the $\N$-step Moser-Tardos algorithm for $(\cal G, \mathbf R,\I,\pi,f)$ converge with probability $1$ to a Borel function which satisfies $\mathbf R$. As always, probability is with respect to the choice of the random seed $\rnd\ssin b^{p\sstimes \N}$

If $\De\ge 2$ then  $\MT_0, \MT_1,\ldots$ also converge with probability $1$ to a Borel function which satisfies $\mathbf R$, by Theorem~\ref{bll}.
\end{proof}

\section*{Appendix A. \,\,Some standard lemmas}
\setcounter{section}{1}
\setcounter{theorem}{0}
\renewcommand{\thesection}{\Alph{section}}
We provide this appendix only for reader's convenience, as all the results are standard.  We start with the following lemma.
\blemm[borel-arrow]
Let $\Om$ be a standard Borel space and let $(U,\ga)$ be a Borel arrow on $\Om$. There exists a partition of $U$ into disjoint Borel sets $U_0,U_1,\ldots$, such that for $x\in U_0$ we have $\ga(x) = x$ and for all $i>0$ and $x\in U_i$ we have $\ga(x) \notin U_i$.
\elemm

\bpf By Kuratowski's theorem on standard Borel spaces we might as well assume that $\Om \subset [0,1)$. Let $d\colon \Om\times\Om\to [0,1)$ be the standard metric on $\Om$.  Let $U_0 := \{x\in U\colon \ga(x) = x\}$, and let $V:= U\setminus U_0$. Let $V_n := \{x\in V\colon d(x,\ga(x)) > \frac{1}{n}\}$. Let us fix $n$ and let $W := V_n$. Clearly it is enough to partition $W$ into countably many sets $W_i$, $i=0,1, \ldots$ such that for $x\in W_i$ we have $\ga(x)\notin W_i$.

For $j\in n$ let $A_j := [\frac{j}{n},\frac{j+1}{n})$. Note that $\ga(A_j)\cap A_j = \emptyset$. Thus we also have $\ga(A_j\sscap W) \cap A_j\sscap W = \emptyset$, and we may define $W_j := A_j\sscap W$.
\epf

\bcory[cacy]
Let $\cal G$ be a locally finite Borel graph without self-loops. There exists a sequence $S = ((V_0,\ga_0),(V_1,\ga_1),\ldots) $ of Borel arrows on $V(\cal G)$ such that $\cal G = \cal G(S)$, and with the property that for all $i\in \N$ we have $N_{\cal G}(U_i)\cap U_i = \emptyset$.
\ecory

\bpf
The previous lemma, together with the assumption about the self loops, show that we can find a sequence $S = ((U_0,\ga_0), (U_1,\ga_1),\ldots)$ of Borel arrows such that $\cal G = \cal G(S)$, and such that for all $i\in \N$ we have $U_i\cap \ga_i(U_i) = \emptyset$.

Let $\mathbf P$ be the set of all finite sets of natural numbers. Let $M\colon V(\cal G)\to \mathbf P$ be defined as follows. For $x\in V(\cal G)$ we let $M(x)$ be the set of those $i$ such that $x\in U_i$ and there is no $j<i$ such that $x\in U_j$ and $\ga_j(x) = \ga_i(x)$. It is not difficult to check that $M$ is a Borel function (i.e.~for $s\in \mathbf P$ we have that $M^{-1}(s)$ is a Borel set).

Now for $i\in \N$ and $s\in \mathbf P$ we define $U(i,s) := U_i \cap M^{-1}(s)$ and let $\ga(i,s)$ be the restriction of $\ga_i$ to $U(i,s)$. Clearly the sets $U(i,s)$ are Borel, and the family $\wb S$ of Borel arrows $((U(i,s),\ga(i,s)))$, $i\in \N$, $s\in \mathbf P$ has the property that $\cal G = \cal G(S) = \cal G(\wb S)$.

Let us argue that for $x\in U(i,s)$ we have $N_{\cal G}(x)\,\,\cap\, U(i,s) = \emptyset$. By the definition of $M$ we have $U(i,s)\subset \bigcap_{j\in s} U_j$  and  $N_{\cal G} (x) = \{\ga_j(x)\colon j \in s\}$. Now the lemma follows because for all $i\in \N$ we have $U_i\cap \ga_i(U_i) = \emptyset$.
\epf

Let $\cal G$ be Borel graph and let $\cal B$ be the Borel $\si$-algebra of $V(\cal G)$. Recall that an independence function for $\cal G$ is a function $\I\colon \cal B \to \cal B$ such that for $X\in \cal B$ the set $\I(X)$ is a maximal subset of $X$ which is independent in $\cal G$.

\blemm[dupa2] Let $\cal G$ be a locally finite Borel  graph. There exists an independence function $\I$ for $\cal G$.
\elemm
\bpf Let us first note that that Lemma~\ref{borel-arrow} implies in particular that the operation of removing all self-loops from $\cal G$ results in a \textit{Borel} graph. Note also that if $\I$ is an independence function for ``$\cal G$ with self-loops removed'' then it is also an independence function for $\cal G$. Thus we may just as well assume that $\cal G$ does not have self-loops.

Furthermore, we have that $\I$ is an independence function for $\cal G$ if and only if it is an independence function for the symmetrization of $\cal G$, so we may assume that $\cal G$ is symmetric.

If $V(\cal G)$ is countable we can use the following greedy algorithm: first we fix an enumeration $x_0,x_1,\ldots$ of $V(\cal G)$ and for $X\subset V(\cal G)$  we define $\I(X)$ inductively as follows. We let $X_0$ be equal to $\{x_0\}\cap X$. Assume that $X_i$ is defined for some $i\in \N$ and let $X_{i+1} := X_i \cup ((\{x_{i+1}\}\cap X) \setminus N(X_i))$.  Finally we let $\I(X) := \bigcup_{i\in \N} X_i$. It is clear that  $\I$ has all the required properties.

Corollary~\ref{cacy} allows us to adapt this  algorithm to an arbitrary Borel graph $\cal G$. Indeed, let $S = ((V_0,\ga_0),(V_1,\ga_1),\ldots)$ be a sequence of Borel arrows as in Corollary~\ref{cacy}, and let $X\in \cal B$. We let $X_0 := V_0\cap X$ and we assume that $X_i$ is defined for some $i\in \N$. Now we define $X_{i+1}: = X_i \cup  ((V_{i+1}\cap X) \setminus N (X_i))$ and $\I(X):= \bigcup_{i\in \N} X_i$. 

It is clear that $\I(X)$ is a Borel subset of $X$. The independence of $\I(X)$ follows from Corollary~\ref{cacy}. Indeed it is enough to inductively show that the sets $X_0,X_1,\ldots$ are independent sets.  For $X_0$ it follows directly from Corollary~\ref{cacy}, so let assume that we we know it for some $i$.

Consider the graph on the vertex set $X_{i+1}$  induced from $\cal G$. We need to show that it has no edges. By definition we have $X_{i+1}: = X_i \cup  ((V_{i+1}\cap X) \setminus N (X_i))$. There are no edges between points of $X_i$ and $X_i$ by the inductive assumption. There are also no edges between points of $(V_{i+1}\cap X) \setminus N (X_i)$ and $(V_{i+1}\cap X) \setminus N (X_i)$ by Corollary~\ref{cacy}. Finally there are clearly no edges between $X_i$ and $(V_{i+1}\cap X) \setminus N (X_i)$. This finishes the proof of the independence of $\I(X)$.

 Let us finally argue about the maximality of $\I(X)$. Suppose there is $x\in X\setminus \I(X)$  such that $\I(X)\cup \{x\}$ is independent. Let $i\in \N$ be the smallest natural number such that $x\in V_{i}$. If $i=0$ then clearly $x\in X_0$. Thus let us assume $i>0$. Since $\I(X)\cup \{x\}$ is independent we have $x \notin N(X_{i-1})$ and so $x \in X_{i}$. Since  $X_i\subset \I(X)$ we have $x\in \I(X)$ which finishes the proof.
\epf

\blemm[dupa3]
Let $\cal G$ be a Borel graph such that $\sup_{x\in V(\cal G)} |N(x)| <\infty$ and let $r\in \N$. There exists a Borel partition of $V(\cal G)$ which is $r$-sparse.
\elemm
\bpf
Let $\cal H$ be the Borel graph defined as follows. We let $V(\cal H) := V(\cal G)$ and we let $(x,y)\in E(\cal H)$ if the distance between $x$ and $y$ in $\cal G$ is at most $n$. Let $\I$ be an independence function for $\cal H$.

We proceed to inductively define a sequence of pair-wise disjoint Borel subsets $W_0,W_1,\ldots$ of $V(\cal H)$ as follows: let $W_0 := \I(V(\cal G))$, and let us assume that for some $i\in \N$ all the sets $W_0, \ldots, W_{i-1}$ are defined. Then we let
$$
    W_i := \I(V(\cal G)\setminus \bigcup_{j=0}^{i-1} W_j).
$$

Clearly it is enough to argue that for some $n$ we have $\bigcup_{j=0}^{n-1} W_j = V(\cal G)$. Let $m := \sup_{x\in V(\cal G)} |B_{\cal H}(x,2)|$. By the assumption on $\cal G$ we have $m<\infty$, and we claim that $\bigcup_{j=0}^{m-1} W_j = V(\cal G)$.

To show it, it is enough to argue that for every $x\in V(\cal H)$ and every $j\in m$ we have that either (i) $W_j\cap B_{\cal H}(x,2) \neq \emptyset$, or (ii) $B_{\cal H}(x,1) \subset  \bigcup_{i=0}^{j-1} W_i$.

If for some $j$ the condition (ii) does not hold then  clearly 
$$
    V(\cal G)\setminus \bigcup_{j=0}^{i-1} W_j \quad \cap \quad B_{\cal H}(x,1) \neq \emptyset.
$$
But  $\I(V(\cal G)\setminus \bigcup_{j=0}^{i-1} W_j)$ is a maximal independent set of $ V(\cal G)\setminus \bigcup_{j=0}^{i-1} W_j$ and hence must contain a point in  $B_{\cal H}(x,2)$, as claimed. This finishes the proof.

\epf

\bibliographystyle{plain}
\bibliography{bibliografia}

\end{document}